\documentclass[a4paper,11pt]{article}
\usepackage{amsmath}
\usepackage{amsfonts}
\usepackage{caption}
\usepackage{float}
\usepackage{subcaption}
\usepackage[latin9]{inputenc}
\usepackage[english]{varioref}
\usepackage{dcolumn}
\usepackage[height=24cm , width = 18cm , top = 2cm , left = 1.5cm, right = 1.5cm,, a4paper]{geometry}
\usepackage[a4paper]{geometry}
\usepackage{graphicx}
\usepackage{pstricks}
\usepackage{multirow}
\usepackage{tabularx}
\usepackage{psfrag}
\usepackage{rotating}
\usepackage{booktabs}

\usepackage{delarray}
\usepackage{amssymb}
\usepackage{blindtext}
\usepackage{rotating}
\usepackage{layout}
\usepackage{amsthm}
\theoremstyle{definition}
\newtheorem{exmp}{Example}[section]
\usepackage{xcolor}
\usepackage{setspace}

{
	{
		{
			{

				\newtheorem{thm}{Theorem}[section]

				\newtheorem{rem}[thm]{Remark}

				\parindent 0pt

\begin{document}
\thispagestyle{empty}

\author{Sanjiv Kumar Bariwal${}^1$\footnote{Email: p20190043@pilani.bits-pilani.ac.in}, Saddam Hussain${}^2$, Rajesh Kumar${}^3$\footnote{{\it{${}$ Corresponding Author. Email Address:}} rajesh.kumar@pilani.bits-pilani.ac.in}\\
\footnotesize ${}^{1,2,3}$Department of Mathematics, Birla Institute of Technology and Science Pilani,\\ \small{ Pilani-333031, Rajasthan, India}\\
}
\date{}										
										
\title{{ Non-linear collision-induced breakage equation: finite volume and semi-analytical methods}}	
									
\maketitle

\begin{quote}
{\small {\em\bf Abstract}}: The non-linear collision-induced breakage equation has significant applications in particulate processes. Two semi-analytical techniques, namely homotopy analysis method (HAM) and accelerated homotopy perturbation method (AHPM) are investigated along with the well-known finite volume method (FVM) to comprehend the dynamical behavior of the non-linear system, i.e., the concentration function, the total number and the total mass of the particles in the system. The theoretical convergence analyses of the series solutions of HAM and AHPM are discussed. In addition, the error estimations of the truncated solutions of both methods equip the maximum absolute error bound. To justify the applicability and accuracy of these methods, numerical simulations are compared with the findings of FVM and analytical solutions considering three physical problems. 
\end{quote}
	{\bf{keywords:}} Collision-induced Breakage, Semi-analytical Method, Series Solution, Convergence, Error Estimates
	\section{Introduction}
	Particulate processes such as aggregation (coagulation), breakage (fragmentation) play an important role in industrial fields and natural phenomena like the formation of droplets \cite{fang2017population}, the distribution of asteriods \cite{wheeler2019effects}, milling \cite{capece2018population}, and crushing in pharamacology to name only a few. Particle dimension is the primary characteristic of the particulate process that is defined by size or volume. The size of the particles is not typically uniform, but it is a distributed property. In the last few decades, the linear breakage equation has gained popularity in several scientific disciplines for analyzing complex engineering problems \cite{ziff1991new,ziff1992explicit}. To improve the quality of analysis for a range of processes, the expansion of the linear model is necessary. Therefore, the incorporation of the non-linearity in the breakage process was considered, that helped in changing the way and rate of the particle breakage behavior. The collision-induced breakage equation (CBE) allows mass distribution between colliding particles that relates to technological applications such as attrition of particles in fluidization \cite{matsuda2004modeling}, nanoparticle production in emulsions \cite{ferrante2005time} and solid liquid slurry flow \cite{peng2020solid}, where mass distribution and size reduction of particles are caused by collisions between them.\\

	The continuous CBE is interest of this present work; Cheng and Redner \cite{cheng1988scaling} used the following integro-partial differential equation to derive the model. It illustrates the time progression of  concentration function  $f(\varsigma,\epsilon)\geq 0$ of particles of size $\epsilon \in {\mathbb{R}}^{+}$ at time $\varsigma\geq 0$ and is defined by
	 \begin{equation}\label{maineq}
		  \left.\begin{aligned}
		        	\frac{\partial{f(\varsigma,\epsilon)}}{\partial \varsigma}= & \int_0^\infty\int_{\epsilon}^{\infty} K(\rho,\sigma)b(\epsilon,\rho,\sigma)f(\varsigma,\rho)f(\varsigma,\sigma)\,d\rho\,d\sigma -\int_{0}^{\infty}K(\epsilon,\rho)f(\varsigma,\epsilon)f(\varsigma,\rho)\,d\rho,\\
		        		f(0,\epsilon)\ \ =& \ \ f^{in}(\epsilon) \geq 0, \ \ \ \epsilon \in{\mathbb{R}}^{+},
		        \end{aligned}
		  \right\}
		 	 \end{equation}
where $(\varsigma,\epsilon)$ are dimensionless independent variables, without losing any generality, and $f$ is an unknown function. The collision kernel $K(\rho,\sigma)$ in Eq.(\ref{maineq}) demonstrates the rate of successful collision for breakage between two particles of sizes $\rho$ and $\sigma$, whereas $b(\epsilon,\rho,\sigma)$ is the breakage rate for construction of an $\epsilon$ size particle from a $\rho$ size particle due to its collision with a $\sigma$ size particle. $K$ satisfies the following  property in terms of practical, i.e.,
\begin{align*}
K(\rho,\sigma)=K(\sigma,\rho) \geq 0,\quad (\rho,\sigma) \in {\mathbb{R}}^{+} \times {\mathbb{R}}^{+},
\end{align*}
 and $b(\epsilon,\rho,\sigma)$ holds some physically meaningful quantities
 \begin{align}
 b(\epsilon, \rho, \sigma)\neq 0\,\,\, \text{for} \hspace{0.4cm} \epsilon \in (0,\rho)\,\,\, \text{and} \hspace{0.4cm} b(\epsilon, \rho, \sigma)=0 \,\,\,\text{for} \hspace{0.4cm} \epsilon> \rho,
 \end{align}
 as well as
 \begin{align}
 \int_{0}^{\rho}\epsilon b(\epsilon, \rho, \sigma)\,d\epsilon=\rho, \,\, \forall\, (\rho,\sigma) \in {\mathbb{R}}^{+} \times {\mathbb{R}}^{+},\\
  \int_{0}^{\rho} b(\epsilon, \rho, \sigma)\,d\epsilon=\bar{N}(\rho, \sigma)< \infty.
 \end{align}
 The implication of Eq.(3) is that the total mass of smaller particles that break off from a particle of size $\rho$ is equal to  $\rho$. The function $\bar{N}(\rho, \sigma)$ in Eq.(4) is the number of particles resulting from the collision of a single particle of size  $\rho$ with a particle of size $\sigma$ and it is evident that this number is greater than 2, see \cite{ernst2007nonlinear}.\\
 
 Moments are used to characterize the concentration function and provide insight into the behaviour of a population within the Population Balance Equations (PBE) framework. Moments of the concentration function are defined as integrals of the concentration function multiplied by the particular powers of the particle size. The $n^{th}$ moment of the concentration function is defined as:
 \begin{align}\label{moment}
 M_{n}(\varsigma)=\int_{0}^{\infty}{\epsilon}^{n}f(\varsigma, \epsilon)\,d\epsilon, \,\, n=0,1,2,\cdots.
 \end{align}
The zeroth (${M_0}$) and first ($M_1$) moments correspond to the total number of particles and total mass of particles, respectively, in the dynamical system. The second moment ($M_2$) expresses the energy dissipation of the system. In this work, we have considered the set of kernels for which CBE (\ref{maineq}) possesses the mass conservation, i.e., $ M_{1}(\varsigma)= M_{1}(0)$ for all $\varsigma>0.$ One property of CBE is the shattering transition \cite{cheng1990kinetics,kostoglou2000study} which characterizes the loss of the particle's mass by converting into dust particles, i.e., $ M_{1}({\varsigma})< M_{1}(0)$ for all $\varsigma>{\varsigma}_s>0$. These integral properties also describe the evolution of the population over time.

\subsection{Literature review and motivation}
In this work, we desire to get the approximate solutions of CBE  (\ref{maineq})  using three different techniques which include a numerical technique (FVM) and two different semi-analytical schemes (HAM, AHPM). Article \cite{cheng1990kinetics} analyzed the asymptotic behaviour of a simple-minded class of models in which  two-particle collision results in either: (1) both particles splitting into two equal pieces, (2) only the larger particle splitting in two, or (3) only the smaller particle splitting.  Next, authors in \cite{krapivsky2003shattering} studied the shattering phenomena and also investigated case (2), where particles turn into dust particles due to discontinuous transition. Moreover,  case (3) contains the continuous transition with dust gaining mass steadily due to the fragments.  The article \cite{kostoglou2000study} possessed the information regarding the analytical solutions for two cases, $ K(\epsilon, \rho)=1, b(\epsilon, \rho, \sigma)=2/\rho$ and $ K(\epsilon, \rho)=\epsilon\rho, b(\epsilon, \rho, \sigma)=2/\rho$ with monodisperse initial condition $\delta(\epsilon-1).$ Additionally, self-similar solutions are also explored for sum kernel $K(\epsilon, \rho)={\epsilon}^{\omega}+{\rho}^{\omega}$ and product kernel $K(\epsilon, \rho)={\epsilon}^{\omega}{\rho}^{\omega}$, $\omega>0$. Further,
the existence of classical solution with mass conservation for coagulation and CBE is  investigated for collision kernels growing indefinitely for large volumes and binary breakage distributon function  in \cite{barik2020global}. In the continuation, authors in \cite{giri2021existence} discussed the existence of mass conserving weak solution for collision kernel $K(\epsilon, \rho)={\epsilon}^{\alpha}{\rho}^{\beta}+{\epsilon}^{\beta}{\rho}^{\alpha},\, \alpha \leq \beta \leq 1,$ when $\alpha+\beta \in [1,2].$\\

        The study of the well-posedness of CBE  is exhibited only for very specific kernels. In the sense of the applicability of CBE, numerical methods such as FVM \cite{das2020approximate, paul2023moments} and finite element method \cite{lombart2022fragmentation} were investigated recently. FVM was proven to be one of the best algorithms to solve such models, see \cite{filbet2004numerical,bariwal2023convergence} and further citations for aggregation, breakage and aggregation-breakage equations. The study of weighted FVM has been accomplished with the event-driven constant number
                Monte Carlo simulation algorithm for several breakage distribution functions in \cite{das2020approximate}. In \cite{paul2023moments}, two new weighted FVM are introduced to witness the preservation of total number of particles and mass. In addition, convergence analysis and consistency are discovered under some assumptions of collisional kernels and initial condition. \\
                
On the other side, semi-analytical methods are emerging as an analytic tool to solve complex non-linear problems, see \cite{liao1995approximate,liao2004homotopy,liu2010essence,kalla2012accelerated,hendi2017accelerated} and further citations. These schemes provide recursive coefficient formulas to find explicitly the analytical solution. This analytic result accords well with the numerical results and may be regarded as the solution definition for the non-linear problems. In this work, we have considered two different semi-analytical techniques, known as homotopy analysis method (HAM) and accelerated homotopy perturbation method (AHPM), for solving the CBE. HAM was first developed by Liao  \cite{liao1995approximate,liao2003beyond,liao2004homotopy} for solving the non-linear problems efficiently and accurately. The articles \cite{liao1995approximate,liao2004homotopy} justify the advantages of HAM over other perturbation methods in terms of showing the accuracy of the results. Next, the convergence analysis and the error estimation are discussed and validated considering two non-linear problems in \cite{odibat2010study}. Next, the idea of AHPM \cite{el2012accelerated} is based on HPM in which accelerated polynomials are introduced by the author to compute the approximated series solutions with high rate of convergence, see \cite{kalla2012accelerated}. The improvement in the results are justified considering several non-linear differential equations.\\ 

The PBEs like aggregation and breakage equations are solved using HAM for various benchmark kernels such as aggregation kernels $\beta(\epsilon,\rho)=1,\,\epsilon+\rho,\,\epsilon\rho,\, {\epsilon}^{2/3}+{\rho}^{2/3}$, breakage kernel $b(\epsilon,\rho)=2/\rho $ with initial data $f(0,\epsilon)=e^{-\epsilon},\, \delta(\epsilon-a)$. Interestingly, for some cases, the closed form solutions are obtained which are actually the exact solutions, see \cite{kaur2022approximate}. By using He's polynomial, HPM was established to solve such PBEs in \cite{dutta2018population,kaur2019analytical}. The CBE is the most unexplored area of PBEs to acquire new analytical results. We believe that this is the first attempt to implement such series solutions, i.e., HAM and AHPM, for solving the CBE (\ref{maineq}). The convergence of the approximated solutions along with upper bound estimations of the errors are studied under some physical assumptions on kernels. Further, to show the novelty of our proposed schemes, results for concentration and moments calculated via HAM and AHPM are compared with the findings of FVM and analytical solutions. Three different numerical examples are considered to support the scheme's reliability. \\

The article's structure is as follows: Section \ref{FVM} provides the discretization scheme of FVM for CBE (\ref{maineq}). Section \ref{semianalyical} contributes to the semi-analytical methodologies of HAM, AHPM, and the coefficient form of the series solutions for CBE. In Section \ref{section4}, convergence and error analysis are exhibited. Further, in Section \ref{numericalresult}, approximate solutions of HAM, AHPM and FVM are reported graphically for the concentration function and moments. The effectiveness of the results is confirmed by comparing them to the analytical outcomes. The conclusions are proposed in Section \ref{conclusions}.

\section{ Finite Volume Method}\label{FVM}
 FVM is a well known numerical technique used to solve mainly partial differential equations (PDEs) related to conservation laws \cite{eymard2000finite}. In this part, we begin to investigate the FVM for the solution of Equation (\ref{maineq}).  In order to do so, volume variable ranging from 0 to $\infty$ is restricted to finite region (0,$R$] where $0< R <\infty$. Hence, the CBE for the truncated domain is given by 
	 \begin{equation}\label{maineq1}
			  \left.\begin{aligned}
			        	\frac{\partial{f(\varsigma,\epsilon)}}{\partial \varsigma}= & \int_0^R\int_{\epsilon}^{R} K(\rho,\sigma)b(\epsilon,\rho,\sigma)f(\varsigma,\rho)f(\varsigma,\sigma)\,d\rho\,d\sigma -\int_{0}^{R}K(\epsilon,\rho)f(\varsigma,\epsilon)f(\varsigma,\rho)\,d\rho,\\
			        		f(0,\epsilon)\ \ =& \ \ f^{in}(\epsilon) \geq 0, \ \ \ \epsilon \in (0,R].
			        \end{aligned}
			  \right\}
			 	 \end{equation}
			 	 Assume a subdivision of the functioning domain (0,$R$] into tiny cells as $\Lambda_i^h:=]\epsilon_{i-1/2}, \epsilon_{i+1/2}], \,\,i=1,2,\ldots,  \mathrm{I}$, where $\epsilon_{1/2}=0, \ \ \epsilon_{\mathrm{I}+1/2}= R, \hspace{0.2cm} \Delta \epsilon_i=\epsilon_{i+1/2}-\epsilon_{i-1/2}.$ Every grid cell has a representative as $\epsilon_i=\frac{\epsilon_{i-1/2}+\epsilon_{i-1/2}}{2}$, the midpoint of cell, and the mean value of the concentration function $f(\varsigma,\epsilon)$ in $i^{th}$ cell is taken by the following equation
			 	 	\begin{align}\label{meandens}
			 	 							f_{i}(\varsigma)=\frac{1}{\Delta \epsilon_i}\int_{\epsilon_{i-1/2}}^{\epsilon_{i+1/2}}f(\varsigma,\epsilon)\,d\epsilon.
			 	 						\end{align}
		We use the following procedure to construct the discretized form of the CBE (\ref{maineq1}): Integrating Eq.(\ref{maineq1}) with regard to $\epsilon$ across the $i^{th}$ cell provides the semi-discrete form as	 	 	
		\begin{align}\label{semi}
		\frac{df_i(\varsigma)}{d\varsigma}=B_{f}(i)-D_{f}(i),
		\end{align}
		where
		\begin{align*}
		B_{f}(i)=\frac{1}{\Delta \epsilon_i}\int_{\epsilon_{i-1/2}}^{\epsilon_{i+1/2}}\int_0^{\epsilon_{\mathrm{I}+1/2}}\int_{\epsilon}^{\epsilon_{\mathrm{I}+1/2}} K(\rho,\sigma)b(\epsilon,\rho,\sigma)f(\varsigma,\rho)f(\varsigma,\sigma)d\rho\,d\sigma\,d\epsilon,
		\end{align*}
		\begin{align*}
		D_{f}(i)= \frac{1}{\Delta \epsilon_i}\int_{\epsilon_{i-1/2}}^{\epsilon_{i+1/2}}\int_0^{\epsilon_{\mathrm{I}+1/2}} K(\epsilon,\rho)f(\varsigma,\epsilon)f(\varsigma,\rho)d\rho\,d\epsilon,
		\end{align*}
		along with initial distribution,
		\begin{align}
		f_{i}(0)=f_{i}^{in}=\frac{1}{\Delta \epsilon_i}\int_{\epsilon_{i-1/2}}^{\epsilon_{i+1/2}}f(0,\epsilon)\,d\epsilon.
		\end{align}	
		Based on the idea of FVM \cite{eymard2000finite}, particle concentration is approximated over a grid cell instead of point mass, i.e., $\hat{f}_{i} \approx f_i$ in the $i^{th}$ cell. Hence,
applying the quadrature rule to each of the representations mentioned above, the semi-discrete equation due to FVM is as follows
			\begin{align}\label{semidiscrete}
		\frac{d\hat{f}_{i}(\varsigma)}{d\varsigma}=&\frac{1}{\Delta \epsilon_i}\sum_{l=1}^{\mathrm{I}}\sum_{j=i}^{\mathrm{I}}K_{j,l}\hat{f}_{j}(\varsigma)\hat{f}_{l}(\varsigma)\Delta \epsilon_{j}\Delta \epsilon_{l}\int_{\epsilon_{i-1/2}}^{\lambda_{j}^{i}}b(\epsilon,\epsilon_{j},\epsilon_{l})\,d\epsilon-\sum_{j=1}^{\mathrm{I}}K_{i,j}\hat{f}_{i}(\varsigma)\hat{f}_{j}(\varsigma)\Delta \epsilon_{j}, 
		\end{align}
		where the term $\lambda_{j}^{i}$ is expressed by
		\begin{equation}
		\lambda_{j}^{i} =
		\begin{cases}
		\epsilon_{i}, & \text{if }\,j=i, \\
		\epsilon_{i+1/2}, & j\neq i.							
		\end{cases}
		\end{equation}	
	For solving the semi-discrete equation (\ref{semidiscrete}), any higher order numerical scheme (ode45) can be used to obtain the approximation results for particular cases of parameters $K$ and $b$. Note that, our motive here is to validate the series solutions using FVM and exact solutions, so a detailed convergence analysis is not explained here. Now, series approximation methods HAM and AHPM are discussed below for the CBE (\ref{maineq}). 	
	\section{Semi-analytical Methods (SAM)}\label{semianalyical}
	These are mathematical techniques that incorporate analytical and numerical approaches to solve complex problems. They  reduce computational complexity by employing analytical expressions for specific problem components, resulting in faster calculations and fewer computational resources. If problems are solvable, SAM provides either the closed form solution which is actually the exact solution or the finite term series solution which is an approximate solution. By providing approximate solutions with known precision, these methods permit researchers to compare the performance and accuracy of various numerical methods, thereby determining their efficacy.
		This section goes into details about the HAM and AHPM, which are well-established SAM. 
	\subsection{Homotopy analysis method}
	The basic preliminaries and applicability of HAM to general form of different equations are discussed in \cite{liao2004homotopy}. We first examine the HAM and then expand it to construct an approximate series solution to CBE (\ref{maineq}). Consider the general functional equation
	\begin{align}\label{nonlinear}
	N[f(\varsigma,\epsilon)]=0,
	\end{align}
where $N$ is a non-linear operator and  $f$ is an unknown function. Proceeding further, by following the idea of HAM, let us construct the zero-order deformation equation as
\begin{align}\label{deformation}
(1-p)L[\phi(\varsigma,\epsilon;p)-f^{in}(\epsilon)]=p \alpha N[\phi(\varsigma,\epsilon;p)],
\end{align}
where, significant roles are played by the embedding parameter $p\in[0,1]$, the non-zero convergence-control parameter $\alpha$, and the auxiliary linear operator $L$ with $L(0)=0$  to modify and manage the convergence domain of series solution \cite{liao2004homotopy}. The term $\phi(\varsigma,\epsilon;p)$ is an unknown function and $f^{in}(\epsilon)$ is an initial guess. Eq.(\ref{deformation}) provides $\phi(\varsigma,\epsilon;0)=f^{in}(\epsilon)$ and $\phi(\varsigma,\epsilon;1)=f(\varsigma,\epsilon),$ when $p=0$ and $p=1,$ respectively. This means that  that the initial guess $f^{in}(\epsilon)$ converts into the solution $f(\varsigma,\epsilon)$ as $p$ varies from 0 to 1. Moreover, the Taylor series expansion of $\phi(\varsigma,\epsilon;p)$ with respect to $p$ yields
\begin{align}\label{phiequation}
\phi(\varsigma,\epsilon;p)=f^{in}(\epsilon)+\sum_{m=1}^{\infty}f_m(\varsigma,\epsilon)p^{m},
\end{align}
where
 \begin{align}\label{solutionham}
f_m(\varsigma,\epsilon)=\frac{1}{m!}\frac{{\partial}^{m}\phi(\varsigma,\epsilon;p)}{\partial p^m}\bigg\vert_{p=0}.
\end{align}
The series solution is obtained by the power series (\ref{phiequation}) that is convergent at $p=1$ and  for a suitable convergence-control parameter $\alpha$, i.e.,
\begin{align}
f(\varsigma,\epsilon)=f^{in}(\epsilon)+\sum_{m=1}^{\infty}f_m(\varsigma,\epsilon).
\end{align}
The unknown terms $f_m(\varsigma,\epsilon)$ are computed with the assistance of high order deformation outlined below. Consider the vector $\vec{f}_{m}=\{f_0,f_1,\ldots,f_m\}$. Differentiating Eq.(\ref{deformation}) $m$ times with respect to $p$,  dividing it by $m!$ with $p=0$ provide the $m$th-order deformation equation as
\begin{align}\label{mthorder}
L[f_m-{\Pi}_mf_{m-1}]=\alpha \Gamma_m(\vec{f}_{m-1}),
\end{align}
where
\begin{align}
\Gamma_m(\vec{f}_{m-1})=\frac{1}{(m-1)!}\frac{{\partial}^{m-1}N[\phi(\varsigma,\epsilon;p)]}{\partial p^{m-1}}\bigg\vert_{p=0},
\end{align}
and 
\begin{equation}
{\Pi}_m =
\begin{cases}
0, & \text{if }\,m\leq 1, \\
1, & m>1.							
\end{cases}
\end{equation}
Liao \cite{liao2003beyond} explains that the so-called generalized Taylor series supplies a way to regulate and adjust the convergence region via an auxiliary parameter, enabling the HAM  particularly well-suited for problems with strong non-linearity. Note that the convergence-control parameter $\alpha$ in Eq.(\ref{mthorder}) can be evaluated using the discrete averaged residual error formula \cite{singh2018analytical}, which is written at the end of this part.\\

Proceeding further, the development of mathematical formulation of CBE (\ref{maineq}) using HAM is explained now. Integrating (\ref{maineq}) from 0 to $\varsigma$ over $\varsigma$  with $f(0,\epsilon)=f^{in}(\epsilon),$ we obtain the integral form as
\begin{align}\label{nonlinear1}
		N[f(\varsigma,\epsilon)]=&f(\varsigma,\epsilon) -f^{in}(\epsilon)-\int_{0}^{\varsigma}\int_0^\infty\int_{\epsilon}^{\infty} K(\rho,\sigma)b(\epsilon,\rho,\sigma)f(\varsigma,\rho)f(\varsigma,\sigma)\,d\rho\,d\sigma\,d\varsigma \nonumber \\ &+\int_{0}^{\varsigma}\int_{0}^{\infty}K(\epsilon,\rho)f(\varsigma,\epsilon)f(\varsigma,\rho)\,d\rho\,d\varsigma.
\end{align}
The zero-order deformation equation of (\ref{nonlinear1}) is
\begin{align}\label{deformation1}
(1-p)[\phi(\varsigma,\epsilon;p)-f^{in}(\epsilon)]=p \alpha N[\phi(\varsigma,\epsilon;p)], \quad p\in [0,1],
\end{align}
where $\phi(\varsigma,\epsilon;p)$ is an unknown function and $N[\phi(\varsigma,\epsilon;p)]$ is constructed as
\begin{align}\label{nonlinear2}
		N[\phi(\varsigma,\epsilon;p)]=&\phi(\varsigma,\epsilon;p) -f^{in}(\epsilon)-\int_{0}^{\varsigma}\int_0^\infty\int_{\epsilon}^{\infty} K(\rho,\sigma)b(\epsilon,\rho,\sigma)\phi(\varsigma,\rho;p)\phi(\varsigma,\sigma;p)\,d\rho\,d\sigma\,d\varsigma \nonumber \\ &+\int_{0}^{\varsigma}\int_{0}^{\infty}K(\epsilon,\rho)\phi(\varsigma,\epsilon;p)\phi(\varsigma,\rho;p)\,d\rho\,d\varsigma=0.
\end{align}
As we have noticed from Eqs.(\ref{phiequation}-\ref{solutionham}) for $p=1,$ the series solution is
\begin{align}
\phi(\varsigma,\epsilon;1)=f(\varsigma,\epsilon)=f^{in}(\epsilon)+\sum_{m=1}^{\infty}f_m(\varsigma,\epsilon).
\end{align}
Next, define the $m$th-order deformation equation as
\begin{align}\label{mthorder1}
f_m-{\Pi}_mf_{m-1}=\alpha \Gamma_m(\vec{f}_{m-1}),
\end{align}
where
\begin{align}
\Gamma_m(\vec{f}_{m-1})=\frac{1}{(m-1)!}\frac{{\partial}^{m-1}N[\phi(\varsigma,\epsilon;p)]}{\partial p^{m-1}}\bigg\vert_{p=0}=\frac{1}{(m-1)!}\frac{{\partial}^{m-1}N\big[\sum_{k=1}^{\infty}f_k(\varsigma,\epsilon) p^k\big]}{\partial p^{m-1}}\bigg\vert_{p=0}.
\end{align}
After simplifying the above term, we have
\begin{align}\label{simplifiedgamma}
\Gamma_m(\vec{f}_{m-1})=&f_{m-1}(\varsigma,\epsilon)-(1-\Pi_m)f^{in}(\epsilon)-\int_{0}^{\varsigma}\int_0^\infty\int_{\epsilon}^{\infty} K(\rho,\sigma)b(\epsilon,\rho,\sigma)\Omega_{m-1}^{1}\,d\rho\,d\sigma\,d\varsigma \nonumber \\
&+\int_{0}^{\varsigma}\int_{0}^{\infty}K(\epsilon,\rho)\Omega_{m-1}^{2}\,d\rho\,d\varsigma,
\end{align}
where 
\begin{align}\label{omegaequation}
\Omega_{m-1}^{1}=\frac{1}{(m-1)!}\frac{{\partial}^{m-1}}{\partial p^{m-1}}\bigg(\sum_{r=1}^{\infty}\sum_{k=0}^{r-1}f_k(\varsigma,\rho)f_{r-k-1}(\varsigma,\sigma)p^k\bigg)\bigg\vert_{p=0},
\end{align}
and
\begin{align}\label{omegaequation1}
\Omega_{m-1}^{2}=\frac{1}{(m-1)!}\frac{{\partial}^{m-1}}{\partial p^{m-1}}\bigg(\sum_{r=1}^{\infty}\sum_{k=0}^{r-1}f_k(\varsigma,\epsilon)f_{r-k-1}(\varsigma,\rho)p^k\bigg)\bigg\vert_{p=0}.
\end{align}
The simplified $m$th-order deformation equation (\ref{mthorder1}) is attained via applying the Eqs.(\ref{simplifiedgamma}-\ref{omegaequation1}),  that is 
\begin{align}\label{solutionterm}
f_m=&{\Pi}_mf_{m-1}+\alpha\bigg[f_{m-1}-(1-{\Pi}_m)f^{in}(\epsilon)-\int_{0}^{\varsigma}\int_0^\infty\int_{\epsilon}^{\infty} K(\rho,\sigma)b(\epsilon,\rho,\sigma)\Omega_{m-1}^{1}\,d\rho\,d\sigma\,d\varsigma \nonumber \\
&+\int_{0}^{\varsigma}\int_{0}^{\infty}K(\epsilon,\rho)\Omega_{m-1}^{2}\,d\rho\,d\varsigma\bigg], \quad m=1,2,3,\cdots.
\end{align}
Hence, the series solution components are derived by (\ref{solutionterm}) with initial condition $f_0(\varsigma,\epsilon,\alpha)=f^{in}(\epsilon),$ and  are listed below 
 \begin{equation}\label{solutionterm1}
	  \left.\begin{aligned}
	         f_0(\varsigma,\epsilon,\alpha)=&f^{in}(\epsilon),\\
	          f_1(\varsigma,\epsilon,\alpha)=&\alpha\bigg[-\int_{0}^{\varsigma}\int_0^\infty\int_{\epsilon}^{\infty} K(\rho,\sigma)b(\epsilon,\rho,\sigma)\Omega_{0}^{1}\,d\rho\,d\sigma\,d\varsigma+\int_{0}^{\varsigma}\int_{0}^{\infty}K(\epsilon,\rho)\Omega_{0}^{2}\,d\rho\,d\varsigma\bigg],\\
	            f_{m}(\varsigma,\epsilon,\alpha)=& (1+\alpha)f_{m-1}+\alpha\bigg[-\int_{0}^{\varsigma}\int_0^\infty\int_{\epsilon}^{\infty} K(\rho,\sigma)b(\epsilon,\rho,\sigma)\Omega_{m-1}^{1}\,d\rho\,d\sigma\,d\varsigma  \\
	            &+\int_{0}^{\varsigma}\int_{0}^{\infty}K(\epsilon,\rho)\Omega_{m-1}^{2}\,d\rho\,d\varsigma\bigg], \quad m> 1.
	        \end{aligned}
	  \right\}
	 	 \end{equation}
For the numerical simulations, let us denote the approximated solution of $n$th-order for CBE (\ref{maineq}) by  
	\begin{align}\label{final1}
	\Theta_{n}(\varsigma,\epsilon,\alpha):=\sum_{i=0}^{n}f_{i}(\varsigma,\epsilon,\alpha).
	\end{align}
The optimal value of $\alpha$, for the best approximated series solution of CBE, is computed by minimizing the following function, i.e.,
	\begin{align}\label{alphavalue}
	\min_{\alpha \in R}A(\alpha), \quad  A(\alpha)=\frac{1}{n^2}\sum_{m=1}^{n}\sum_{r=1}^{n}[N(\Theta_{n}(\varsigma_m,\epsilon_r,\alpha))]^2,
	\end{align}
	where $A(\alpha)$ is the averaged residual error of the $n$th-order approximation for CBE problem and $(\varsigma_m,\epsilon_r)$ belongs to the  operational domain. This concept's specifics can be extracted from \cite{singh2018analytical}. 
		The symbolic computation software MATHEMATICA has a command \lq\lq Minimize\rq\rq \ to compute the optimal value of $\alpha$ explicitly.
	\subsection{Accelerated homotopy perturbation method}
In HPM \cite{he2003homotopy}, He's polynomial is equivalently taken as Adomain polynomial introduced in \cite{adomian1994solving}. The author in \cite{el2007error} generates a new class of accelerated polynomials which after incorporating into HPM yields a new scheme called AHPM that provides better accuracy to solve non-linear differential equations. To explain AHPM, consider a general functional form as
\begin{align}\label{AHPM1}
\psi(\varsigma,\epsilon)-	N[\psi(\varsigma,\epsilon)]=g(\epsilon),
\end{align}
where $(\varsigma,\epsilon)$ are independent variables, $\psi$ is an unknown function and $g$ is a given function. Also, $N$ is the non-linear operator. Subsequently, rewrite Eq.(\ref{AHPM1}) with solution $f(\varsigma,\epsilon)= \psi(\varsigma,\epsilon)$ into the following form
	\begin{align}\label{AHPM2}
	\mathrm{L}(f(\varsigma,\epsilon))=f(\varsigma,\epsilon)-g(\epsilon)-	N[f(\varsigma,\epsilon)].
	\end{align}
Further, construct a homotopy $H(f,p)$ having the properties $H(f,0)=T(f), \, H(f,1)=\mathrm{L}(f(\varsigma,\epsilon))$ and
\begin{align}\label{AHPM3}
H(f,p)=(1-p)T(f)+p\mathrm{L}(f)=0,
\end{align}	
where $T(f)$ is an operator with solution $f_0$. The embedding parameter $p$ is crucial for continuously deforming Eq.(\ref{AHPM3}) from $H(f_0,0)=T(f_0)=0$ to $ H(\psi,1)=\mathrm{L}(f)=0$ as $p$ increases from 0 to 1. The semi-analytical solution $f(\varsigma,\epsilon)=\sum_{n=0}^{\infty}p^{n}f_{n}(\varsigma,\epsilon)$  and exact solution $\psi(\varsigma,\epsilon)$ have a relation as
\begin{align}\label{AHPM4}
\psi(\varsigma,\epsilon)=\sum_{n=0}^{\infty}f_{n}(\varsigma,\epsilon)=\lim_{p\rightarrow 1}f(\varsigma,\epsilon).
\end{align}
	For further use, let us suppose $T(f)=f(\varsigma,\epsilon)-g(\epsilon)$. Substituting this and $\mathrm{L}(f)$ from Eq.(\ref{AHPM2}) into Eq.(\ref{AHPM3}) yields a form
	\begin{align}\label{AHPM5}
	H(f,p)=f-g-pN(f)=0.
	\end{align}
	The non-linear term $N(f)$ can be  written in the form of accelerated polynomials ($\hat{H}_{n}$) that is explained in \cite{el2007error}, as 
	\begin{align}\label{AHPM6}
	N(f)=\sum_{n=0}^{\infty}p^n\hat{H}_{n}(f_0,f_1,\cdots,f_n),
	\end{align}
	where $\hat{H}_{n}$ has mathematical structure as 
	\begin{align}\label{Hvalue}
	\hat{H}_{n}(f_0,f_1,\cdots,f_n)=N(S_n)-\sum_{i=0}^{n-1}\hat{H}_{i}, \quad n\geq 1,
	\end{align}
with $S_n=\sum_{i=0}^{n}f_{i}(\varsigma,\epsilon)$  and $\hat{H}_{0}=N(f_0).$ The series iterative terms are obtained by substituting $f(\varsigma,\epsilon)=\sum_{n=0}^{\infty}p^{n}f_{n}(\varsigma,\epsilon)$  and (\ref{AHPM6}) into Eq.(\ref{AHPM5}) and then by comparing the powers of $p$,  
	 \begin{equation}\label{solutionterm2}
		  \left.\begin{aligned}
		         f_0(\varsigma,\epsilon)=& f^{in}(\epsilon)=g(\epsilon),\\
		            f_{n}(\varsigma,\epsilon)=& \hat{H}_{n-1}, \quad n\geq 1.
		        \end{aligned}
		  \right\}
		 	 \end{equation}
		 	Next, to find the mathematical formulation for CBE (\ref{maineq}) by AHPM, the above idea follows. Therefore, integrating Eq.(\ref{maineq}) over time variable $\varsigma$  provides   
		 	 \begin{align}
		 	 \mathrm{L}(f(\varsigma,\epsilon))=&f(\varsigma,\epsilon) -f^{in}(\epsilon)-\int_{0}^{\varsigma}\int_0^\infty\int_{\epsilon}^{\infty} K(\rho,\sigma)b(\epsilon,\rho,\sigma)f(\varsigma,\rho)f(\varsigma,\sigma)\,d\rho\,d\sigma\,d\varsigma \nonumber \\ &+\int_{0}^{\varsigma}\int_{0}^{\infty}K(\epsilon,\rho)f(\varsigma,\epsilon)f(\varsigma,\rho)\,d\rho\,d\varsigma,
		 	 \end{align}
	where
	\begin{align}
	N[f(\varsigma,\epsilon)]=\int_{0}^{\varsigma}\int_0^\infty\int_{\epsilon}^{\infty} K(\rho,\sigma)b(\epsilon,\rho,\sigma)f(\varsigma,\rho)f(\varsigma,\sigma)\,d\rho\,d\sigma\,d\varsigma  -\int_{0}^{\varsigma}\int_{0}^{\infty}K(\epsilon,\rho)f(\varsigma,\epsilon)f(\varsigma,\rho)\,d\rho\,d\varsigma.
	\end{align}
	Now, supersede $f(\varsigma,\epsilon)=\sum_{n=0}^{\infty}p^{n}f_{n}(\varsigma,\epsilon)$ and (\ref{AHPM6}) into (\ref{AHPM5}), yield a homotopy. Proceeding further to  equate the terms with identical powers of p lead to get the iterative components, i.e.,
	\begin{align}
	\sum_{n=0}^{\infty}p^nf_{n}(\varsigma,\epsilon)-f^{in}(\epsilon)-p\hat{H}_{n}(f_0,f_1,\cdots,f_n)=0,
	\end{align}
where for $n=0$, it provides the first term
	\begin{align}\label{AHPMseriesterm1}
	 f_0(\varsigma,\epsilon)=& f^{in}(\epsilon),
	\end{align}
for $n=1$
\begin{align}\label{AHPMseriesterm2}
f_{1}(\varsigma,\epsilon)=\int_{0}^{\varsigma}\int_0^\infty\int_{\epsilon}^{\infty} K(\rho,\sigma)b(\epsilon,\rho,\sigma)f_0(\varsigma,\rho)f_0(\varsigma,\sigma)\,d\rho\,d\sigma\,d\varsigma  -\int_{0}^{\varsigma}\int_{0}^{\infty}K(\epsilon,\rho)f_0(\varsigma,\epsilon)f_0(\varsigma,\rho)\,d\rho\,d\varsigma,
\end{align}
 and for $n\geq 2$, the general term is
 \begin{align}\label{AHPMseriesterm3}
 f_{n}(\varsigma,\epsilon)=&\int_{0}^{\varsigma}\int_0^\infty\int_{\epsilon}^{\infty} K(\rho,\sigma)b(\epsilon,\rho,\sigma)\left(\sum_{i=0}^{n-1}f_i(\varsigma,\rho)\right)\left(\sum_{i=0}^{n-1}f_i(\varsigma,\sigma)\right)\,d\rho\,d\sigma\,d\varsigma \nonumber\\ &-\int_{0}^{\varsigma}\int_{0}^{\infty}K(\epsilon,\rho)\left(\sum_{i=0}^{n-1}f_i(\varsigma,\epsilon)\right)\left(\sum_{i=0}^{n-1}f_i(\varsigma,\rho)\right)\,d\rho\,d\varsigma-\left(\sum_{i=1}^{n-1}f_i(\varsigma,\epsilon)\right).
 \end{align}
 The truncated AHPM series solution of $n$th-order is expressed by the term  $\Upsilon_{n}(\varsigma,\epsilon)$, i.e.,
 \begin{align}
 \Upsilon_{n}(\varsigma,\epsilon):=\sum_{i=0}^{n}f_{i}(\varsigma,\epsilon).
 \end{align}
 \section{Convergence Analysis}\label{section4}
  This section describes the reliability and efficiency of the methods with the given framework. It enables the development of theoretical guarantees concerning the algorithm's behavior. Here, convergence analysis of the HAM and AHPM solutions shall be illustrated for a set of assumptions on collisional kernels. For that, consider the set $\mathcal{W}=\{(\varsigma,\epsilon): 0\leq \varsigma \leq T, 0<\epsilon<\infty\}$ for a fix $T$ and assume a space of all continuous function $f$, say  $\mathbb{Y}_{r,s}(T)$ with the following induced norm
 	\begin{align}\label{norm}
 	\|f\|=\sup_{\varsigma\in [0,T]}\int_{0}^{\infty}\Big({\epsilon}^r+\frac{1}{{\epsilon}^{2s}}\Big) |f(\varsigma,\epsilon)| d\epsilon,\quad r\geq 1, s\geq 0.
 	\end{align} 
 	Eq.(\ref{nonlinear1}) provides the new operator form as
 	\begin{align}\label{operator1}
 	f=\mathcal{S}f,
 	\end{align}
 		where $\mathcal{S}:\mathbb{Y}_{r,s}(T)\rightarrow \mathbb{Y}_{r,s}(T) $ is a non-linear operator given as
 		\begin{align}\label{operator2}
 		\mathcal{S}f=f^{in}(\epsilon)+{\mathcal{L}}^{-1}\Bigl[\int_0^\infty\int_{\epsilon}^{\infty} K(\rho,\sigma)b(\epsilon,\rho,\sigma)f(\varsigma,\rho)f(\varsigma,\sigma)\,d\rho\,d\sigma-\int_{0}^{\infty}K(\epsilon,\rho)f(\varsigma,\epsilon)f(\varsigma,\rho)\,d\rho\Bigr],
 		\end{align}
 		with $\mathcal{L}^{-1}=\int_{0}^{\varsigma}(\cdot)d\varsigma$.
 		Firstly, for the existence of the solution in $\mathbb{Y}_{r,s}^{+}(t)$ (set of non-negative functions from $\mathbb{Y}_{r,s}(T)$), operator $\mathcal{S}$ has to fulfill the contractive property under some hypotheses using the following theorem.
 		\begin{thm}\label{thm1}
 			Assume that the non-linear operator ${\mathcal{S}}$ is defined in (\ref{operator2}). If the following hypotheses;
 			\begin{description}
 			\item[(a)]$K(\epsilon,\rho)$ is non-negative and continuous function with compact support $\mathcal{K}_1=\sup_{\frac{s}{R}\leq \epsilon,\rho\leq R }K(\epsilon,\rho),\,(\epsilon,\rho)\in[0,R]\times [0,R]$,
 			\item[(b)] $b(\epsilon,\rho,\sigma)$ is non-negative, continuous function satisfying the condition $$\int_{0}^{\infty}{\epsilon}^{-\theta s}b(\epsilon,\rho,\sigma)d \epsilon \leq \eta{\rho}^{-\theta s},\quad \theta,\eta >0,$$
 				\item[(c)] $L_0$ is a fix constant for a small $t>0$ such that
 				\begin{align}
 					\|f\|_t=\sup_{\varsigma\in [0,t]}\int_{0}^{\infty}\Big({\epsilon}^r+\frac{1}{{\epsilon}^{2s}}\Big) |f(\varsigma,\epsilon)| d\epsilon\leq L_0,
 				\end{align}
 					\end{description}
 			hold, then the operator  ${\mathcal{S}}$ has contractive nature, i.e., $\|{\mathcal{S}}f-{\mathcal{S}}f^{\star}\|\leq \xi \|f_-f^{\star}\|, \forall \,  (f,f^{\star}) \in \mathbb{Y}_{r,s}^{+}(t) \times \mathbb{Y}_{r,s}^{+}(t),$ where $\xi = 2t\mathcal{K}_1(\mu+1)L_0<1$ with $\mu=\max\{\bar{N},\eta\}$.
 		\end{thm}
 		\begin{proof}
 		The detailed proof of this theorem had been investigated in [\cite{kharchandy2023note}, Theorem 1].
 		\end{proof}
 		\begin{thm}
 		Let all the assumptions of Theorem \ref{thm1} hold. If $	\Theta_{m}=\sum_{i=0}^{m}f_{i}$ is the $m$th-order HAM truncated solution obtained using (\ref{solutionterm1}-\ref{final1}) for the CBE (\ref{maineq}), then $\Theta_{m}$ converges to the exact solution $f(\varsigma,\epsilon)$ having following error bound
 		\begin{align}
 		\|f-\Theta_{m}\|\leq \frac{\nabla^m}{1-\nabla}\|f_1\|,
 		\end{align}
 		where $\nabla=\xi|\alpha|+|1+\alpha|<1$ and $\|f_1\|<\infty.$
 		\end{thm}
 	\begin{proof}
 	 From Eqs.(\ref{solutionterm1}-\ref{final1}), we have 
 		\begin{align}\label{nthorderham}
 		\Theta_{m}=\sum_{i=0}^{m}f_{i}=&f^{in}(\epsilon)-\alpha\sum_{j=1}^{m}{\mathcal{L}^{-1}}\bigg[\int_0^\infty\int_{\epsilon}^{\infty} K(\rho,\sigma)b(\epsilon,\rho,\sigma)\Omega_{j-1}^{1}\,d\rho\,d\sigma  
 			            -\int_{0}^{\infty}K(\epsilon,\rho)\Omega_{j-1}^{2}\,d\rho\bigg]\nonumber\\
 			            &+(1+\alpha)\sum_{j=1}^{m-1}f_{j}\nonumber\\
 			             =&f^{in}(\epsilon)-\alpha{\mathcal{L}^{-1}}\bigg[\int_0^\infty\int_{\epsilon}^{\infty} K(\rho,\sigma)b(\epsilon,\rho,\sigma)\Big(\sum_{j=1}^{m}\Omega_{j-1}^{1}\Big)\,d\rho\,d\sigma  
 			            			            -\int_{0}^{\infty}K(\epsilon,\rho)\Big(\sum_{j=1}^{m}\Omega_{j-1}^{2}\Big)\,d\rho\bigg]\nonumber\\
 			            			            &+(1+\alpha)\sum_{j=1}^{m-1}f_{j}. 
 		\end{align}
 		
 			Adding and subtracting the term $(1+\alpha)f^{in}(\epsilon)$ reduce the Eq.(\ref{nthorderham}) into the following one
 		\begin{align}
 			\Theta_{m}=&-\alpha \bigg[f^{in}(\epsilon)+{\mathcal{L}^{-1}}\bigg[	\underbrace{\int_0^\infty\int_{\epsilon}^{\infty} K(\rho,\sigma)b(\epsilon,\rho,\sigma)\Big(\sum_{j=1}^{m}\Omega_{j-1}^{1}\Big)\,d\rho\,d\sigma  
 			 -\int_{0}^{\infty}K(\epsilon,\rho)\Big(\sum_{j=1}^{m}\Omega_{j-1}^{2}\Big)\,d\rho}_{\Omega(m)}\bigg]\bigg]\nonumber\\
 			 &+(1+\alpha)	\Theta_{m-1}\nonumber\\
 			=& -\alpha[f^{in}(\varsigma)+{\mathcal{L}^{-1}}[\Omega(m)]]+(1+\alpha)	\Theta_{m-1}.
 		\end{align}
 		For all $n,m \in \mathbb{N}$ and $n>m$, it is certain that
	\begin{align*}
	\|\Theta_{n}-\Theta_{m}\|\leq |\alpha|\|\big[f^{in}(\varsigma)+{\mathcal{L}^{-1}}[\Omega(n)]-f^{in}(\varsigma)-{\mathcal{L}^{-1}}[\Omega(m)]\big]\|+|1+\alpha|\|\Theta_{n-1}-	\Theta_{m-1}\|.
	\end{align*}
	Article \cite{rach2008new} suggested that $\sum_{i=1}^{n}\Omega_{i-1}^{1}\leq f(\Theta_{n-1}(\rho))f(\Theta_{n-1}(\sigma))$ and $\sum_{i=1}^{n}\Omega_{i-1}^{2}\leq f(\Theta_{n-1}(\epsilon))f(\Theta_{n-1}(\rho)),$ implies that
		\begin{align*}
		\|\Theta_{n}-\Theta_{m}\|\leq& |\alpha| \|\big[f^{in}(\varsigma)+{\mathcal{L}^{-1}}[\Theta_{n-1}]-f^{in}(\varsigma)-{\mathcal{L}^{-1}}[\Theta_{m-1}]\big]\|+|1+\alpha|\|\Theta_{n-1}-	\Theta_{m-1}\|\nonumber \\
		=& |\alpha| \|\big[\mathcal{S}\Theta_{n-1}-\mathcal{S}\Theta_{m-1}\big]\|+|1+\alpha|\|\Theta_{n-1}-	\Theta_{m-1}\|.
		\end{align*}
	Using the contractive result of $\mathcal{S}$ in the aforementioned equation leads to
		\begin{align*}
			\|\Theta_{n}-\Theta_{m}\|\leq& \xi|\alpha|\|\Theta_{n-1}-\Theta_{m-1}\|+|1+\alpha|\|\Theta_{n-1}-\Theta_{m-1}\| \nonumber \\
			=&\nabla\|\Theta_{n-1}-\Theta_{m-1}\|,
		\end{align*}
	
	for $\nabla=\xi|\alpha|+|1+\alpha|$. After substituting $n=m+1$, the above inequality converted into
	\begin{align}\label{thetavalue}
		\|\Theta_{m+1}-\Theta_{m}\|\leq \nabla	\|\Theta_{m}-\Theta_{m-1}\|\leq \nabla^2	\|\Theta_{m-1}-\Theta_{m-2}\|\leq \cdots \leq \nabla^m\|\Theta_{1}-\Theta_{0}\|.
	\end{align}
	Eq.(\ref{thetavalue}) assists in finding the bound of $ \|\Theta_{n}-\Theta_m\|$ having triangle inequality, as
	\begin{align}
		\|\Theta_{n}-\Theta_{m}\|\leq&	\|\Theta_{m+1}-\Theta_{m}\|+\|\Theta_{m+2}-\Theta_{m+1}\|+\cdots +\|\Theta_{n-1}-\Theta_{n-2}\|+\|\Theta_{n}-\Theta_{n-1}\|\nonumber\\
		\leq& \nabla^m\big(1+\nabla+\nabla^2+\cdots+\nabla^{n-m-1}\big)\|\Theta_{1}-\Theta_{0}\|=\nabla^m\bigg(\frac{1-\nabla^{n-m}}{1-\nabla}\bigg)\|f_1\|.
			\end{align}
			If $\nabla<1,$ then $(1-\nabla^{n-m})<1$ and $\|f_1\|<\infty$, thus we get 
			\begin{align}
				\|\Theta_{n}-\Theta_{m}\|\leq \frac{\nabla^m}{1-\nabla}\|f_1\|,
			\end{align}
	which converges to zero as $m\rightarrow \infty$. Therefore, a function $\Theta$ exists such that $\lim_{n\rightarrow \infty}\Theta_{n}=\Theta$. Hence, $f=\sum_{i=0}^{\infty}f_i=\lim_{n\rightarrow \infty}\Theta_{n}=\Theta$, which is the exact solution of (\ref{maineq}). For a fixed $m$ and taking $n\rightarrow \infty$, the error bound is derived as
				\begin{align*}
			 		\|f-\Theta_{m}\|\leq \frac{\nabla^m}{1-\nabla}\|f_1\|,
			 		\end{align*}
			 		where $f_1$ is given in (\ref{solutionterm1}).
 	\end{proof}
 	\begin{rem}
 The value of $\alpha$ is consider such that $\nabla<1$, for that
 \begin{align*}
 \xi|\alpha|+|1+\alpha|<1 \implies \xi< \frac{1-|1+\alpha|}{|\alpha|}, \, \alpha \neq 0.
 \end{align*}
 Hence, to ensure that $\nabla<1$ , $\alpha$ should be chosen from [-1, 0). 
 	\end{rem}
 The following theorem equips the convergence and error results for the CBE's approximated $n$th-order AHPM solutions. 
 \begin{thm}
 	Let us assume that all the assumptions of Theorem \ref{thm1} hold. Consider $	\Upsilon_{m}=\sum_{i=0}^{m}f_{i}$ is the $m$th-order AHPM truncated solution obtained using (\ref{solutionterm2}) for the CBE (\ref{maineq}). Then $\Upsilon_{m}$ converges to the exact solution $f(\varsigma,\epsilon)$ with the following error bound
  		\begin{align}\label{error2}
  		\|f-\Upsilon_m\|\leq \frac{\xi^m}{1-\xi}\|f_1\|,
  		\end{align}
  		where $\xi=2t\mathcal{K}_1(\mu+1)L_0<1$ and $\|f_1\|<\infty.$
 \end{thm}
 \begin{proof}
 From the Eqs.(\ref{Hvalue}) and (\ref{solutionterm2}), $m$th-order approximated solution is 
 \begin{align}\label{truncated}
 \Upsilon_{m}=\sum_{i=0}^{m}f_{i}=f^{in}+\hat{H}_{0}+\hat{H}_{1}+\cdots +\hat{H}_{m-1}
 \end{align}
 where 
 \begin{align*}
 \hat{H}_{0}=&N(f_0), \quad \hat{H}_{1}=N(f_0+f_1)-N(f_0),\\
  \hat{H}_{2}=&N(f_0+f_1+f_2)-N(f_0+f_1),\\
  \hat{H}_{m-1}=&N(f_0+f_1+\cdots+f_{m-1})-N(f_0+f_1+\cdots+f_{m-2}).
 \end{align*}
 Substitute all the values of $\hat{H}_{m}$ into (\ref{truncated}), the  new form of  $\Upsilon_{m}$ is obtained as
 \begin{align}
 \Upsilon_{m}=f^{in}+N(f_0+f_1+\cdots+f_{m-1})=f^{in}+N(\Upsilon_{m-1})=\mathcal{S}\Upsilon_{m-1}.
 \end{align}
 By Theorem \ref{thm1}, we have
 \begin{align}
 	\|\Upsilon_{m+1}-\Upsilon_{m}\|\leq \xi	\|\Upsilon_{m}-\Upsilon_{m-1}\|,
 	\end{align}
 	implies that
 	 \begin{align}
 	 	\|\Upsilon_{m+1}-\Upsilon_{m}\|\leq \xi	\|\Upsilon_{m}-\Upsilon_{m-1}\|\leq \xi^2\|\Upsilon_{m-1}-\Upsilon_{m-2}\|\leq \cdots \leq \xi^m\|\Upsilon_{1}-\Upsilon_{0}\|.
 	 	\end{align}
 	 	Next, the triangle inequality helps us to get the bound of $\|\Upsilon_{n}-\Upsilon_{m}\|$ for all $n,m \in \mathbb{N}$ with $n>m$, as
 	 	\begin{align}
 	 			\|\Upsilon_{n}-\Upsilon_{m}\|\leq&	\|\Upsilon_{m+1}-\Upsilon_{m}\|+\|\Upsilon_{m+2}-\Upsilon_{m+1}\|+\cdots +\|\Upsilon_{n-1}-\Upsilon_{n-2}\|+\|\Upsilon_{n}-\Upsilon_{n-1}\|\nonumber\\
 	 			\leq& \xi^m\big(1+\xi+\xi^2+\cdots+\xi^{n-m-1}\big)\|\Upsilon_{1}-\Upsilon_{0}\|=\xi^m\bigg(\frac{1-\xi^{n-m}}{1-\nabla}\bigg)\|f_1\|.
 	 				\end{align}
 	 				Again, the term $\xi<1$ provides the following
 	 				\begin{align}
 	 					\|\Upsilon_{n}-\Upsilon_{m}\|\leq \frac{\xi^m}{1-\xi}\|f_1\|,
 	 				\end{align}
 	 		that converges to zero as $m\rightarrow \infty$. Therefore, a function $\Upsilon$ exists such that $\Upsilon_{n} \rightarrow \Upsilon$, as $n\rightarrow \infty$ and so $f=\sum_{i=0}^{\infty}f_i=\lim_{n\rightarrow \infty}\Upsilon_{n}=\Upsilon$, which is the exact solution of (\ref{maineq}). Further,  for a fixed $m$ and letting $n\rightarrow \infty$, the error bound (\ref{error2}) is accomplished.			
 \end{proof}
 \section{Numerical Results and Discussion}\label{numericalresult}
 This section analyses the accuracy, verification, and visual representation of the numerical  and semi-analytical methods and their results. 
It contains the findings of CBE (\ref{maineq}) by implementing the FVM, HAM, and AHPM for various collisional kernels. The discussion regarding the concentration function and the integral property of concentration (moments) is supported by the approximate solutions provided by three different schemes for an appropriate time scale. Graphs display all the results simulated by MATLAB and MATHEMATICA software.
\begin{exmp}\label{firstexample}
Assuming Eq.(\ref{maineq}) with  $K(\epsilon,\rho)=\epsilon\rho$, $b(\epsilon,\rho,\sigma)=\frac{2}{\rho}$ and $f^{in}(\epsilon)=\exp(-\epsilon)$ for which the  corresponding exact solution is  $f(\varsigma,\epsilon)=(1+\varsigma)^{2}\exp(-\epsilon(1+\varsigma))$  described in \cite{kostoglou2000study}. 
\end{exmp}
Eq.(\ref{solutionterm1}) provides the series solution terms, thanks to HAM as
 \begin{align*}
	 f_{0}(\varsigma,\epsilon,\alpha)=& e^{-\epsilon},\quad
	      f_{1}(\varsigma,\epsilon,\alpha)=\alpha  \varsigma \left(e^{-\epsilon} \epsilon-2 e^{-\epsilon}\right),\\
	       f_{2}(\varsigma,\epsilon,\alpha)=&\frac{1}{2} \alpha  \varsigma e^{-\epsilon} \left(2 \alpha  (\varsigma-2)+\alpha  \varsigma \epsilon^2+\epsilon (\alpha  (2-4\varsigma)+2)-4\right),\\
	    f_{3}(\varsigma,\epsilon,\alpha)=&\frac{1}{6} \alpha  \varsigma e^{-\epsilon} \bigl(\alpha ^2 \varsigma^2 \epsilon^3+6 \epsilon \bigl(\alpha ^2 \bigl(\varsigma^2-4 \varsigma+1\bigl)+\alpha  (2-4 \varsigma)+1\bigl)+12 (\alpha +1) (\alpha  (\varsigma-1)-1)\\
	    &+6 \alpha  \varsigma \epsilon^2 (\alpha +\alpha  (-\varsigma)+1)\bigl).
	 \end{align*}
	 While, the computation of series terms of approximated AHPM solution is derived by Eqs.(\ref{AHPMseriesterm1}-\ref{AHPMseriesterm3}), we acquire
	  \begin{align*}
	 	 f_{0}(\varsigma,\epsilon)=& e^{-\epsilon},\quad
	 	      f_{1}(\varsigma,\epsilon)=\varsigma \left(2 e^{-\epsilon}-e^{-\epsilon} \epsilon\right), \quad
	 	       f_{2}(\varsigma,\epsilon)=\frac{1}{2} \varsigma^2 e^{-\epsilon} ((\epsilon-4) \epsilon+2),\\
	 	    f_{3}(\varsigma,\epsilon)=&-\frac{1}{6} \varsigma^3 e^{-\epsilon} \epsilon^3+\varsigma^3 e^{-\epsilon} \epsilon^2-\varsigma^3 e^{-\epsilon} \epsilon, \quad
	 	    f_{4}(\varsigma,\epsilon)=\frac{1}{24} \varsigma^4 e^{-\epsilon} \epsilon^4-\frac{1}{3} \varsigma^4 e^{-\epsilon} \epsilon^3+\frac{1}{2} \varsigma^4 e^{-\epsilon} \epsilon^2,\\
	 	   f_{5}(\varsigma,\epsilon) =&-\frac{1}{120} \varsigma^5 e^{-\epsilon} \epsilon^3 ((\epsilon-10) \epsilon+20).
	 	 \end{align*}
 Continuing in this manner, one can  compute the higher order terms of the series solution using MATHEMATICA. The $5$th-order series approximated solutions, i.e., $\Theta_{5}$ of HAM  and $\Upsilon_5$ of AHPM, are considered for evaluation and the results are compared with finite volume and exact solutions. The corresponding optimal value of $\alpha=-0.826$, up to 3 significant digit (SD) with rounding, is generated by applying the formula (\ref{alphavalue}) for the truncated solution of HAM. \\

 The illustrations of the concentration functions obtained via three methods are visualized along  with the analytical solution for various volumes in Fig.{\ref{fig1}}. As anticipated, the concentration function exhibits decreasing behaviour as size increases. At a particular time $\varsigma=1$, small-sized particles are abundant and large-sized particles are scarce. We observe that FVM results demonstrate excellent chemistry with the analytical solution for particles of any sizes, whereas HAM ($\Theta_5$) and AHPM ($\Upsilon_5$) show acceptable precision for particles of small sizes but fail to maintain the accuracy for large sizes particles. Error plots in Fig.{\ref{fig7}} are graphical representations that assess the accuracy and quality of predicated FVM, HAM and AHPM solutions. The maximum error bounds of HAM, AHPM and FVM emerge for small size particles at considerable time and  are almost identical to zero. This indicates the novelty of our proposed schemes. Now, we expressed the concept of the experimental order of convergence (EOC) for the approximated methods to quantify how quickly the error decreases. The formula to find the EOC for FVM, HAM and AHPM is written as follows
\begin{align}\label{exacteoc}
\text{EOC}=\frac{1}{\ln( 2)}\ln \left(\frac{E_{\mathrm{I}}}{E_{2\mathrm{I}}}\right),
\end{align}
where $E_\mathrm{I}$ is the discrete error norm for $\mathrm{I}$ number of mesh points, i.e., $E_\mathrm{I}=\|N_{E}-N_{\mathrm{I}}\|$, where $N_{E}=\sum_{i=1}^{\mathrm{I}}f_{i}\Delta \epsilon_{i}$ and $N_{I}=\sum_{i=1}^{\mathrm{I}}\hat{f}_{i}\Delta \epsilon_{i}$ are the total number of particles generated analytically and numerically via FVM (\ref{semidiscrete}), respectively. Similarly, for HAM $N_{I}=\sum_{i=1}^{\mathrm{I}}\Theta_{n}^{i}\Delta \epsilon_{i}$ and for AHPM $N_{I}=\sum_{i=1}^{\mathrm{I}}\Upsilon_{n}^{i}\Delta \epsilon_{i}$, where $\Theta_{n}^{i}=\Theta_n(\varsigma,\epsilon_i)$ and $\Upsilon_{n}^{i}=\Upsilon_n(\varsigma,\epsilon_i).$ Table \ref{tab:table} reveals that the EOC of all the techniques is one which implies that they are first order accurate. 
	\begin{table}[!htb]
	           	  \centering
	           	    \begin{tabular}{ |p{1.4cm}|p{1cm}|p{1cm}|p{1cm}|}
	           	  
	           	   	            \hline
	           	   	            Cells (I)      & FVM & HAM & AHPM\\
	           	   	                 \hline
	           	   	              30& -  &- &-\\
	           	   	             \hline
	           	   	              60   &  1.1515   &0.9808 &0.9981 \\
	           	   	                 \hline
	           	   	               120  & 1.0901 &0.9906 & 0.9990 \\
	           	   	               \hline
	           	   	               240  & 1.0497 &0.9953 & 0.9994  \\
	           	   	            \hline
	           	           \end{tabular}
	           	            \caption{EOC using FVM, HAM and AHPM at time $\varsigma=1$ for Example {\ref{firstexample}}}
	           	            \label{tab:table}
	           	           \end{table}

Furhter, to justify the approximated solutions, integral properties like moments are also computed by Eq.(\ref{moment}) and compared with the analytical moments. Fig.\ref{fig2} provides information regarding the comparison of HAM, AHPM, FVM moments with the exact ones at time $\varsigma=1.$ The zeroth moment, in Fig.\ref{fig2}(a) can be visualized with an increasing pattern as time increases. All the results of approximated moment overlap to each other and demonstrate excellent accuracy with the exact one. As expected, in Fig.{\ref{fig2}(b)}, the approximated first moments obtained by HAM, AHPM and FVM are constant and  identical to the analytic moment, that is 1. It shows the conservation of mass during the particulate process in the system for a set of kernels considered in Example \ref{firstexample}. The approximated second moments using $\Theta_5$, $\Upsilon_5$ and FVM have decreased behaviour along with the exact moment as time progresses, see Fig.{\ref{fig2}(c)}. HAM and FVM moments have affirmed satisfactory agreement to the precised one, whereas AHPM results deviate after time $\varsigma=0.5.$ 

\begin{figure}[htb!]
 \centering  \includegraphics[width=0.45\textwidth,height=0.35\textwidth]{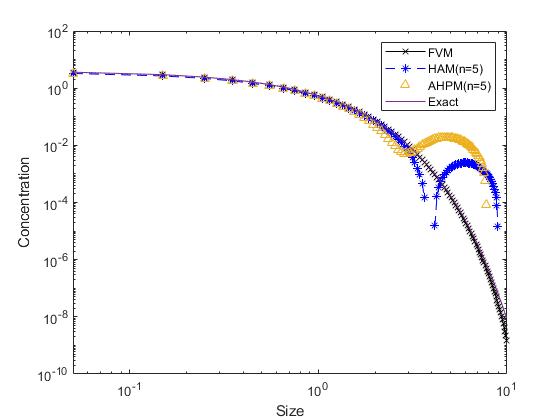}
   \caption{Log-log plots of concentration functions at time 1}
   \label{fig1}
 \end{figure}
\begin{figure}[htb!]
     \centering
     \begin{subfigure}{0.45\textwidth}
         \centering
         \includegraphics[width=\textwidth]{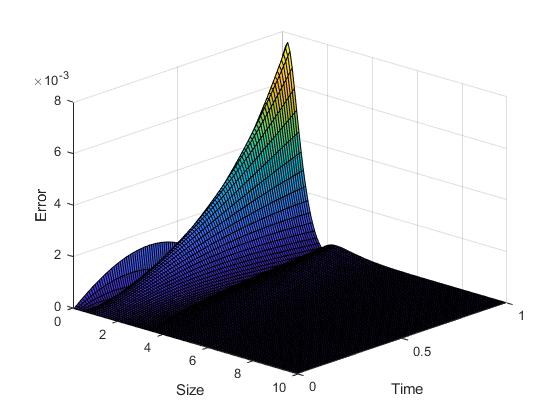}
         \caption{FVM error}
         \label{fig: a1}
     \end{subfigure}
     \hfill
     \begin{subfigure}{0.45\textwidth}
         \centering
         \includegraphics[width=\textwidth]{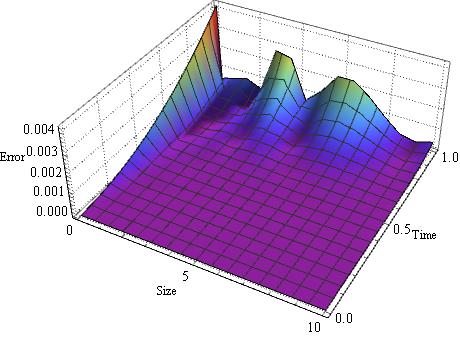}
         \caption{HAM error}
         \label{fig:b11}
     \end{subfigure}
     \hfill
     \begin{subfigure}{0.45\textwidth}
         \centering
         \includegraphics[width=\textwidth]{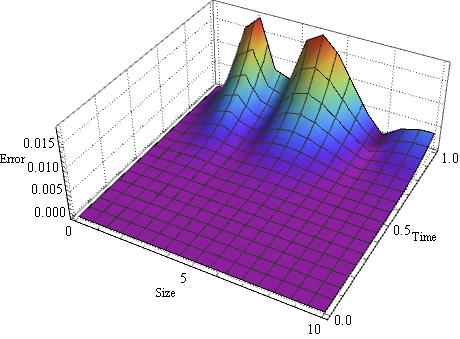}
         \caption{AHPM error}
         \label{fig:c1}
     \end{subfigure}
        \caption{Absolute error plots}
        \label{fig7}
\end{figure}
\begin{figure}[htb!]
     \centering
     \begin{subfigure}{0.45\textwidth}
         \centering
         \includegraphics[width=\textwidth]{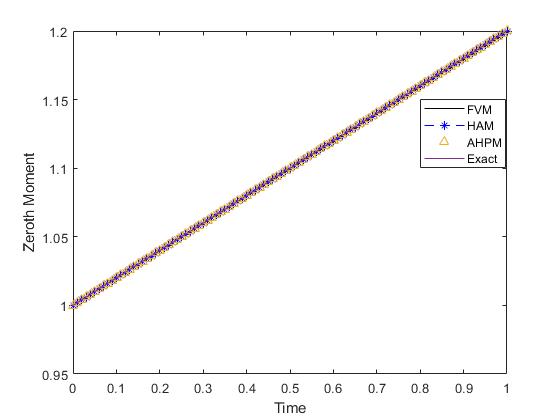}
         \caption{Zeroth moment}
         \label{fig: a}
     \end{subfigure}
     \hfill
     \begin{subfigure}{0.45\textwidth}
         \centering
         \includegraphics[width=\textwidth]{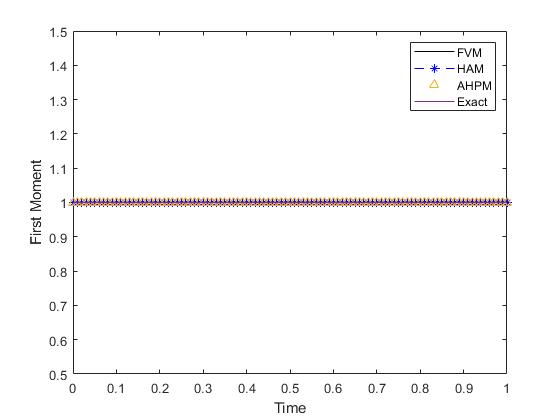}
         \caption{First moment}
         \label{fig:b}
     \end{subfigure}
     \hfill
     \begin{subfigure}{0.45\textwidth}
         \centering
         \includegraphics[width=\textwidth]{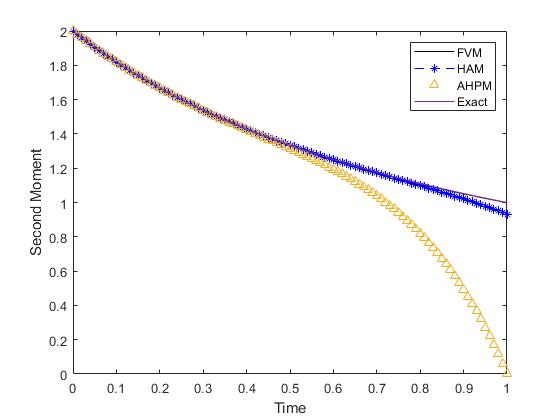}
         \caption{Second moment}
         \label{fig:c}
     \end{subfigure}
        \caption{Moments comparison at time 1: FVM, HAM, AHPM and Exact}
        \label{fig2}
\end{figure}
 \begin{exmp}\label{example2}
	 	 Let us consider  Eq.(\ref{maineq}) with kernels $K(\epsilon,\rho)=\frac{\epsilon\rho}{20}$, $b(\epsilon,\rho,\sigma)=\frac{2}{\rho}$ and initial data  $f^{in}(\epsilon)={\epsilon}\exp(-\epsilon).$ The analytical solution for the concentration is hard to compute, however the precise formulations for the zeroth and  first moments are $M_0(\varsigma)=1+\frac{\varsigma}{5}$ and $M_1(\varsigma)=2,$ respectively. These moments can be easily generated by multiplying Eq.(\ref{maineq}) by 1 and $\epsilon$ as well as integrating from 0 to $\infty$ over $\epsilon$. Substituting these kernel parameters and initial condition in Eq.(\ref{solutionterm1}), we get the series terms due to HAM as
	 	 \end{exmp}	
 \begin{align*}
	 f_{0}(\varsigma,\epsilon,\alpha)=& \epsilon e^{-\epsilon},\quad
	      f_{1}(\varsigma,\epsilon,\alpha)=\frac{1}{10} \alpha  \varsigma e^{-\epsilon} ((\epsilon-2) \epsilon-2), \\
	       f_{2}(\varsigma,\epsilon,\alpha)=&\frac{1}{200} \alpha  \varsigma e^{-\epsilon} (\alpha  (\varsigma \epsilon ((\epsilon-4) \epsilon-2)+4 \varsigma+20 (\epsilon-2) \epsilon-40)+20 ((\epsilon-2) \epsilon-2)),\\
	    f_{3}(\varsigma,\epsilon,\alpha)=&\frac{\alpha ^2 \varsigma^2 e^{-\epsilon} \left(\alpha  \varsigma \epsilon \left((\epsilon-6) \epsilon^2+12\right)+30 (\alpha +1) (\epsilon ((\epsilon-4) \epsilon-2)+4)\right)}{6000}\\
	    &+(\alpha +1) \big(\frac{1}{2} \alpha  \varsigma^2 \big(\frac{1}{100} \alpha  e^{-\epsilon} \epsilon \big(\epsilon^2-2 \epsilon-2\big)-\frac{1}{50} \alpha  e^{-\epsilon} \big(\epsilon^2-2\big)\big)+\alpha  (\alpha +1) \varsigma \big(\frac{1}{10} e^{-\epsilon} \epsilon^2-\frac{1}{5} e^{-\epsilon} (\epsilon+1)\big)\big).
	 \end{align*}
For $ \alpha =-0.969$ taken up to 3 SD with rounding, we have considered $5$th-order HAM $(\Theta_5)$ truncated series solution for numerical verification. The higher terms could not be obtained due to the complexity of the model and parameters. \\
	
	Further, AHPM (\ref{AHPMseriesterm1}-\ref{AHPMseriesterm3}) is also applied for this case to compute the first few series terms for the approximated solution. The terms are listed below
	 \begin{align*}
	 	 f_{0}(\varsigma,\epsilon)=& \epsilon e^{-\epsilon},\quad
	 	      f_{1}(\varsigma,\epsilon)=\frac{1}{10} \varsigma e^{-\epsilon} (2-(\epsilon-2) \epsilon), \quad
	 	       f_{2}(\varsigma,\epsilon)=\frac{1}{200} \varsigma^2 e^{-\epsilon} (\epsilon ((\epsilon-4) \epsilon-2)+4),\\
	 	    f_{3}(\varsigma,\epsilon)=&-\frac{\varsigma^3 e^{-\epsilon} \epsilon^4}{6000}+\frac{\varsigma^3 e^{-\epsilon} \epsilon^3}{1000}-\frac{1}{500} \varsigma^3 e^{-\epsilon} \epsilon, \quad
	 	    f_{4}(\varsigma,\epsilon)=\frac{\varsigma^4 e^{-\epsilon} \epsilon^5}{240000}-\frac{\varsigma^4 e^{-\epsilon} \epsilon^4}{30000}+\frac{\varsigma^4 e^{-\epsilon} \epsilon^3}{60000}+\frac{\varsigma^4 e^{-\epsilon} \epsilon^2}{10000}.
	 	 \end{align*}
In the absence of an exact solution of concentration function, the difference between consecutive terms of series solutions is plotted in Fig.\ref{fig3}(b) to see the precision of semi-analytical algorithms. Graph shows that the difference between 4 and 5 terms solutions is negligible for each schemes and hence, it is certain that the approximated solutions converge to the exact one. Further, to validate the results, $5$th-order truncated solutions via HAM ($\Theta_5)$ and AHPM ($\Upsilon_5$) are also compared with the FVM solution in Fig.\ref{fig3}(a) at time $\varsigma=1.$ The figure elaborates that HAM and AHPM results have little disturbance to FVM solution for small size particles, however, all results exhibit similar behaviour as particle size increases. The approximated concentration plots display an increasing pattern till the size $\epsilon=1$, which decreases further as particle size progresses.\\

Since, expressions for analytical moments are available, to justify the approximated solutions, we resemble the approximated moments of HAM using $\Theta_5$, AHPM via $\Upsilon_5$ and FVM to the corresponding analytical moments in Fig.\ref{fig4} for different time distributions. In Figs.{\ref{fig4}(a), \ref{fig4}(b)}, zeroth and first moments of HAM, AHPM and FVM show remarkable accuracy with the exact ones for various time scales. Therefore, we can consider our approximated solutions to predict better results for different particle sizes and times. In the unavailability of the exact second moment, series (HAM, AHPM) and FVM moments are compared at a time scale of 0 to 1 in Fig.{\ref{fig4}(c)}. The figure indicates that all approximated moments follow same pattern where HAM and AHPM moments exactly coincide with each other. 
\begin{figure}[htb!]
     \centering
     \begin{subfigure}{0.45\textwidth}
         \centering
         \includegraphics[width=\textwidth]{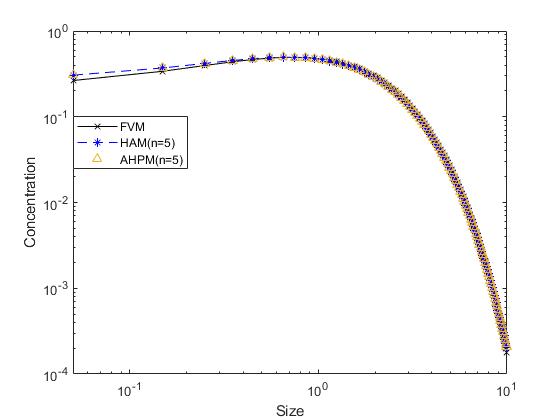}
         \caption{Log-log plots of concentration functions at time 1}
         \label{fig: a11}
     \end{subfigure}
     \hfill
     \begin{subfigure}{0.45\textwidth}
         \centering
         \includegraphics[width=\textwidth]{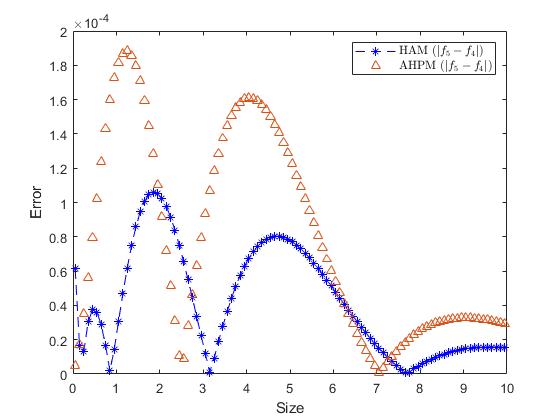}
         \caption{Terms error}
         \label{fig:b1}
     \end{subfigure}
        \caption{Concentration and consecutive terms error plots}
        \label{fig3}
\end{figure}
 	\begin{exmp}\label{exmaple3}
Consider Eq.(\ref{maineq}) with $K(\epsilon,\rho)=1$, $b(\epsilon,\rho,\sigma)=\delta(\epsilon-0.4\rho)+\delta(\epsilon-0.6\rho)$ and $f^{in}(\epsilon)=\exp(-\epsilon).$ In this case, the exact solution for the concentration does not exist in the literature. However, exact moments (zeroth, first and second) can be calculated analytically as $M_0(\varsigma)=\frac{1}{1-\varsigma},\quad  M_1(\varsigma)= 1$ and $M_2(\varsigma)=2(1-\varsigma)^{0.48}.$ The second moment is computed by multiplying Eq.(\ref{maineq}) by ${\epsilon}^2$ and integrating  from 0 to $\infty$ over $\epsilon$.
 	 	  	\end{exmp}
 	 	  	
Implement the HAM, AHPM and FVM to determine a solely approximative solutions. The HAM (\ref{solutionterm1}) is applied for these set of kernels to get the following series terms
\begin{align*}
 	 	  		 	  	 f_{0}(\varsigma,\epsilon,\alpha)=& e^{-\epsilon},\quad
 	 	  		 	  	      f_{1}(\varsigma,\epsilon,\alpha)=\alpha  \varsigma \left(-1.67 e^{-1.67 \epsilon} \theta (0.67 \epsilon)-2.5 e^{-2.5 \epsilon} \theta (1.5 \epsilon)+e^{-\epsilon}\right), \\
 	 	  		 	  	       f_{2}(\varsigma,\epsilon,\alpha)=&\alpha  \varsigma e^{-18.36 \epsilon} \bigl(\theta (0.67 \epsilon) \bigl(1.39 \alpha  \varsigma e^{15.58 \epsilon} \theta (1.11 \epsilon)+2.08 \alpha  \varsigma e^{14.19 \epsilon} \theta (2.5 \epsilon)
 	 	  		 	  	       +e^{16.69\epsilon} (-1.67 \alpha -0.83 \alpha  \varsigma-1.67)\bigl)\\
 	 	  		 	  	      +\theta (1.5 \epsilon)& \bigl(\alpha  \varsigma \bigl(2.08 e^{14.19 \epsilon} \theta (1.67 \epsilon)
 	 	  		 	  	       +3.13 e^{12.11 \epsilon} \theta (3.75 \epsilon)\bigl)+e^{15.86 \epsilon} (-2.5 \alpha -1.25 \alpha  \varsigma-2.5)\bigl)
 	 	  		 	  	       +( \alpha +1) e^{17.36 \epsilon}\bigl),
 	 	  		 	  	 \end{align*}
 	 	  		 	  	 where $\delta$ and $\theta$ are Dirac's delta and Heaviside step functions, respectively.
\begin{figure}[htb!]
     \centering
     \begin{subfigure}{0.45\textwidth}
         \centering
         \includegraphics[width=\textwidth]{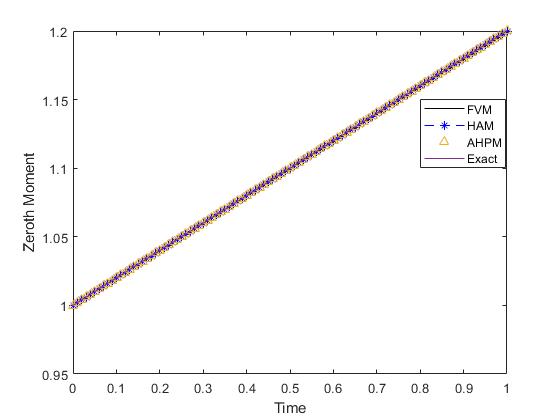}
         \caption{Zeroth moment}
         \label{fig:aa}
     \end{subfigure}
     \hfill
     \begin{subfigure}{0.45\textwidth}
         \centering
         \includegraphics[width=\textwidth]{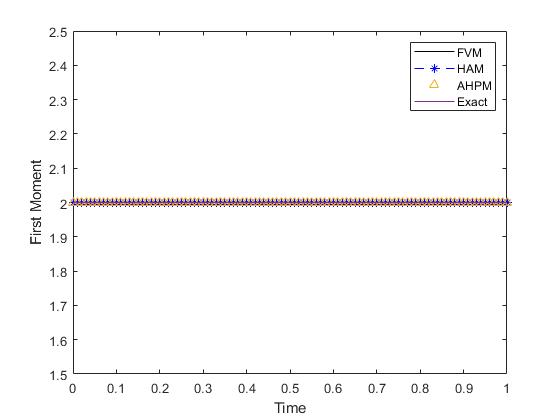}
         \caption{First moment}
         \label{fig:bb}
     \end{subfigure}
     \hfill
     \begin{subfigure}{0.45\textwidth}
         \centering
         \includegraphics[width=\textwidth]{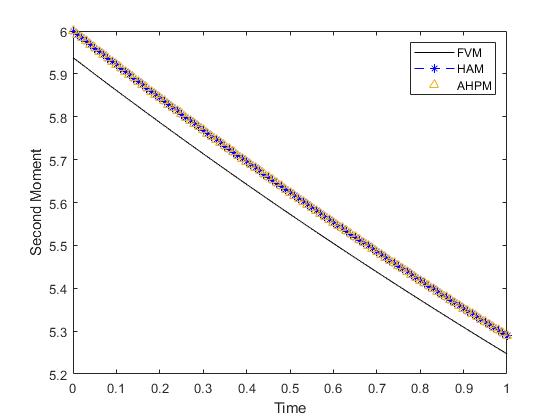}
         \caption{Second moment}
         \label{fig:cc}
     \end{subfigure}
        \caption{Moments comparison at time 1: FVM, HAM, AHPM and Exact}
        \label{fig4}
\end{figure}\\

	 	 We continue this procedure to compute the next term to estimate the $3rd$-order approximate solution $(\Theta_3)$ by MATHEMATICA. The optimal value of $\alpha$ is taken as  -0.829, again using only 3 SD with rounding, by operating Eq.(\ref{alphavalue}). Moving further, the second semi-analytical method AHPM executes the first few terms of the series solution as
	 \begin{align*}
	  f_{0}(\varsigma,\epsilon)=&  e^{-\epsilon},\quad
	  f_{1}(\varsigma,\epsilon)=\varsigma \left(1.67 e^{-1.67 \epsilon} \theta (0.67 \epsilon)+2.5 e^{-2.5 \epsilon} \theta (1.5 \epsilon)-e^{-\epsilon}\right),\\
	  f_{2}(\varsigma,\epsilon)=&\varsigma e^{-22.53 \epsilon} \bigl(\theta (1.5 \epsilon) \bigl(\varsigma e^{18.36 \epsilon} ((2.08\, -1.39 \varsigma) \theta (1.67 \epsilon)+2.78 \varsigma \theta (x))+\varsigma e^{16.28 \epsilon} ((3.12\, -2.08 \varsigma) \theta (3.75 \epsilon)\\
	  &+4.17 \varsigma \theta (\epsilon))+e^{20.03 \epsilon} \bigl(2.5\, -0.83 \varsigma^2\bigl)+((-0.83 \varsigma-1.25) \varsigma-2.5) e^{20.03 \epsilon}\bigl)\\
	  &+\varsigma \theta (0.67 \epsilon) \bigl(e^{19.75 \epsilon} ((1.39\, -0.93 \varsigma) \theta (1.11 \epsilon)+1.85 \varsigma \theta (\epsilon))+e^{18.36 \epsilon} ((2.08\, -1.39 \varsigma) \theta (2.5 \epsilon)+2.78 \varsigma \theta (\epsilon))\\
	  &+(-1.11 \varsigma-0.83) e^{20.86 \epsilon}\bigl)+0.33 \varsigma^2 e^{21.53 \epsilon}\bigl).
	 	 \end{align*}
An approximated truncated solution of order 3 $(\Upsilon_3)$ is considered to see the observation of concentration function and moments for a range of particle sizes.  Due to the lack of an exact concentration function, only approximate solutions of HAM $(\Theta_3)$, AHPM $(\Upsilon_3)$ and FVM are presented graphically in Fig.{\ref{fig5}} for time $\varsigma=0.5$ over a range of particle sizes. Interestingly, FVM and HAM results are almost identical up to size 10, but AHPM does not show the same behaviour as HAM/FVM for large size particles. Similar to the previous case, it is noticed that the difference between consecutive terms of the series solutions decreases and tends to zero, ensuring the convergence of the algorithms. Thus, plot is omitted here. The approximated moments computed via FVM and series solution methods are compared with the analytical moments in Fig.{\ref{fig6}} to justify the novelty of schemes. In Fig.{\ref{fig6}(a)}, zeroth moments using AHPM 3-term solution ($\Upsilon_3$) and FVM are exactly matching with the exact number of particles. However, HAM ($\Theta_3$) moment starts slipping away from the accuracy after time 0.25. In addition, Fig.{\ref{fig6}(c)} also justifies a slight disturbance in the second moment of HAM. As expected, the total mass is conserved and all the approximated methods provide the excellent accuracy with it, see Fig.{\ref{fig6}(b)}. In conclusion, we can visualize from these figures that AHPM and FVM results are nearest to the exact ones as compared to the HAM.
\begin{figure}[htb!]
 \centering
   \includegraphics[width=0.45\textwidth,height=0.35\textwidth]{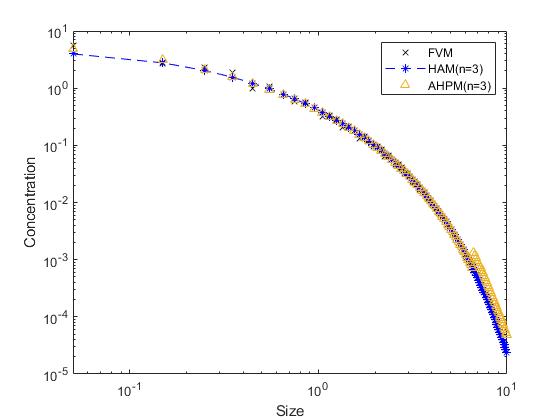}
   \caption{Log-log plots of concentration functions at time 0.5}
   \label{fig5}
 \end{figure}
 	\begin{figure}[htb!]
 	     \centering
 	     \begin{subfigure}{0.45\textwidth}
 	         \centering
 	         \includegraphics[width=\textwidth]{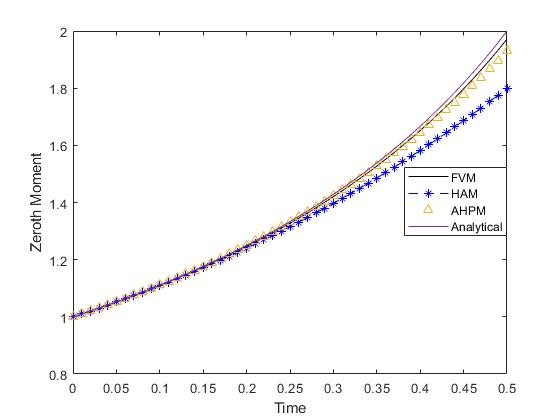}
 	         \caption{Zeroth moment}
 	         \label{fig:aaa}
 	     \end{subfigure}
 	     \hfill
 	     \begin{subfigure}{0.45\textwidth}
 	         \centering
 	         \includegraphics[width=\textwidth]{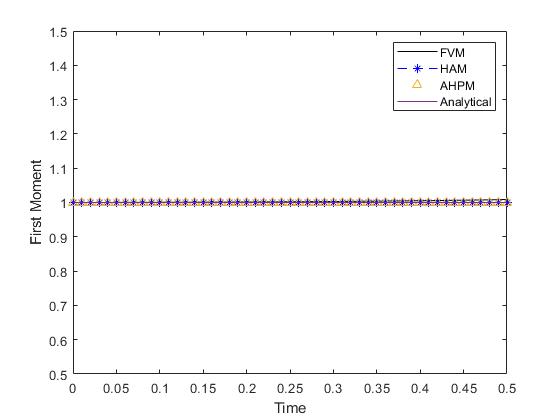}
 	         \caption{First moment}
 	         \label{fig:bbb}
 	     \end{subfigure}
 	     \begin{subfigure}{0.45\textwidth}
 	         \centering
 	         \includegraphics[width=\textwidth]{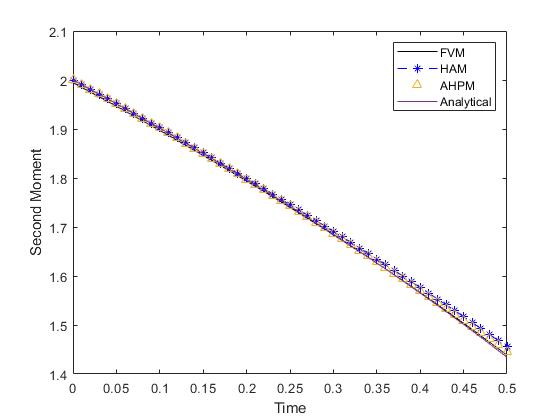}
 	         \caption{Second moment}
 	         \label{fig:ccc}
 	     \end{subfigure}
 	        \caption{Moments comparison at time 0.5: FVM, HAM, AHPM and Exact}
 	        \label{fig6}
 	\end{figure}

\section{Conclusions}\label{conclusions} The HAM, AHPM and FVM
were employed successfully for solving the non-linear collision-induced breakage equation with
certain collision and breakage kernels. The approximate analytical solutions were obtained by truncating the infinite series form of the series solution which was proven to be the exact solution. These methods were easy to implement on such non-linear integro-partial differential equations for various kernels and exponential decay initial functions. Convergence analysis was exhibited for the series solutions of HAM and AHPM and it was reliable enough to achieve the error estimations in each case. The approximated results of concentration function and moments by HAM, AHPM and FVM were compared to the exact ones for three test problems. All the graphs described the good precision and efficiency of considered techniques.\\

{\bf{CONFLICT OF INTEREST:}} This work does not have any conflicts of interest.\\

{\bf{ACKNOWLEDGEMENTS:}} SKB work is supported by CSIR India, which provides the PhD fellowship and the file No. is 1157/CSIR-UGC NET June 2019. The author, SH thankfully acknowledge the financial assistance provided by CSIR, file No. 09/719(0125)/2021 EMR-I, India.

	\bibliography{reference}
	\bibliographystyle{ieeetr}
	\end{document}